\documentclass[12pt, reqno]{amsart}
\usepackage{amsmath, amsthm, amscd, amsfonts, amssymb, graphicx, color, mathrsfs}
\usepackage[bookmarksnumbered, colorlinks, plainpages]{hyperref}
\usepackage[all]{xy}
\usepackage{slashed}

\usepackage{soul}
\usepackage{cancel}
\usepackage{ulem}

\newtheorem{theorem}{Theorem}[section]
\newtheorem{lemma}[theorem]{Lemma}
\newtheorem{proposition}[theorem]{Proposition}

\theoremstyle{definition}

\theoremstyle{remark}
\newtheorem{remark}[theorem]{Remark}
\numberwithin{equation}{section}

\begin{document}
\setcounter{page}{1}

\title[  Nuclear multipliers associated to the Harmonic oscillator ]{ On nuclear $L^p$-multipliers associated to the Harmonic oscillator}

\author[E. S. Barraza]{E. Samuel Barraza}
\address{
 E. Samuel Barraza
  \endgraf
  Department of Computer Science and Artificial Intelligence 
  \endgraf
  Universidad de Sevilla.
  \endgraf
  Sevilla
  \endgraf
  Espa\~{n}a
  \endgraf
  {\it E-mail address} {\rm edglibre@gmail.com; edgbarver@alum.us.es
}
  }
\author[D. Cardona]{Duv\'an Cardona}
\address{
  Duv\'an Cardona:
  \endgraf
  Department of Mathematics  
  \endgraf
  Pontificia Universidad Javeriana.
  \endgraf
  Bogot\'a
  \endgraf
  Colombia
  \endgraf
  {\it E-mail address} {\rm duvanc306@gmail.com;
cardonaduvan@javeriana.edu.co}
  }


\dedicatory{Dedicated to the memory of Professor Guillermo Restrepo Sierra.}



\begin{abstract}
In this paper we  study multipliers associated to the harmonic oscillator (also called Hermite multipliers) belonging to the ideal of $r$-nuclear operators on Lebesgue spaces. We also study the  nuclear trace and the spectral trace of these operators.
\end{abstract} \maketitle

\section{Introduction}
In this paper, we are interested in the $r$-nuclearity of  multipliers associated to the harmonic oscillator (also called Hermite multipliers) on $L^p(\mathbb{R}^n)$-spaces. In quantum mechanics the harmonic oscillator is the unbounded operator defined by $H=-\Delta_x+|x|^2,$ where $\Delta_x$ is the Laplacian. The operator $H$  extends to an unbounded self-adjoint operator on $L^{2}(\mathbb{R}^n) $, and its spectrum consists of the discrete set  $\lambda_\nu:=2|\nu|+n,$ $\nu\in \mathbb{N}_0^n,$ with {real eigenfunctions} $\phi_\nu,$ $\nu\in \mathbb{N}_0^n$, called Hermite functions.
 Every Hermite function  $\phi_{\nu}$  on $\mathbb{R}^n$ has the form
\begin{equation}
\phi_\nu=\Pi_{j=1}^n\phi_{\nu_j},\,\,\, \phi_{\nu_j}(x_j)=(2^{\nu_j}\nu_j!\sqrt{\pi})^{-\frac{1}{2}}H_{\nu_j}(x_j)e^{-\frac{1}{2}x_j^2}
\end{equation} 
where $x=(x_1,\cdots,x_n)\in\mathbb{R}^n$, $\nu=(\nu_1,\cdots,\nu_n)\in\mathbb{N}^n_0,$ and $H_{\nu_j}(x_j)$ denotes the Hermite polynomial of order $\nu_j$.  Furthermore, by the spectral theorem, for every $f\in\mathscr{D}(\mathbb{R}^n)$ we have
\begin{equation}
Hf(x)=\sum_{\nu\in\mathbb{N}^n_0}\lambda_\nu\widehat{f}(\phi_\nu)\phi_\nu(x),\,\,\,\widehat{f}(\phi_\nu) :=\langle f,\phi_\nu \rangle_{L^2(\mathbb{R}^n)}=\int_{\mathbb{R}^n}f(x)\phi_\nu(x)\,dx,
\end{equation} where $\widehat{f}(\phi_v) $ is the Hermite-Fourier transform of $f$ at $\nu.$
In addition,  every function $m$ on $\mathbb{N}_0^n$ has associated a multiplier (or Hermite multiplier) which is a linear operator $T_m$ of the form:
\begin{equation}
T_mf(x)=\sum_{\nu\in\mathbb{N}^n_0}m(\nu)\widehat{f}(\phi_\nu)\phi_\nu(x),\,\,\,\,f\in D(T_m).
\end{equation}
The discrete function $m$ is called the symbol of the operator $T_m.$ In particular, if  $m$ is a measurable function, the symbol of the spectral multiplier $m(H)$, defined by the functional calculus, is given by  $m(\nu):=m(\lambda_\nu)$, thus, spectral multipliers are   natural examples of multipliers associated to the harmonic oscillator.

In order that the operator $T_m:L^{p_1}(\mathbb{R}^n)\rightarrow L^{p_2}(\mathbb{R}^n)$ extends to a $r$-nuclear operator,  we  provide sufficient conditions on the symbol $m.$ We recall the notion of $r$-nuclearity as follows.  By following A. Grothendieck \cite{GRO}, we can recall that a linear operator $T:E\rightarrow F$  ($E$ and $F$ Banach spaces) is  $r$-nuclear, if
there exist  sequences $(e_n ')_{n\in\mathbb{N}_0}$ in $ E'$ (the dual space of $E$) and $(y_n)_{n\in\mathbb{N}_0}$ in $F$ such that
\begin{equation}\label{nuc}
Tf=\sum_{n\in\mathbb{N}_0} e_n'(f)y_n,\,\,\, \textnormal{ and }\,\,\,\sum_{n\in\mathbb{N}_0} \Vert e_n' \Vert^r_{E'}\Vert y_n \Vert^r_{F}<\infty.
\end{equation}
\noindent The class of $r-$nuclear operators is usually endowed with the quasi-norm
\begin{equation}
n_r(T):=\inf\left\{ \left\{\sum_n \Vert e_n' \Vert^r_{E'}\Vert y_n \Vert^r_{F}\right\}^{\frac{1}{r}}: T=\sum_n e_n'\otimes y_n \right\}
\end{equation}
\noindent and, if $r=1$, $n_1(\cdot)$ is a norm and we obtain the ideal of nuclear operators. In addition, when $E=F$ is a Hilbert space and $r=1$ the definition above agrees with the concept of  trace class operators. For the case of Hilbert spaces $H$, the set of $r$-nuclear operators agrees with the Schatten-von Neumann class of order $r$ (see Pietsch  \cite{P,P2}).\\
\\
In order to study the $r$-nuclearity and the spectral trace of Hermite multipliers, we will use results from J. Delgado \cite{D2}, on the characterization of nuclear integral  operators on $L^p(X,\mu)$ spaces, which in this case can be applied to Lebesgue spaces on $\mathbb{R}^n$. Indeed, we will prove that under certain conditions, a $r$-nuclear operator $T_m: L^{p}(\mathbb{R}^n)\rightarrow  L^{p}(\mathbb{R}^n)$ has a nuclear trace given by
\begin{eqnarray}
\textnormal{Tr}(T_m)=\sum_{\nu\in\mathbb{N}^n_0}m(\nu).
\end{eqnarray}
Furthermore, by using the version of the Grothendieck-Lidskii formula proved in O. Reinov, and Q. Latif \cite{O}  we will show that the nuclear trace of these operators coincides with the spectral trace  for $\frac{1}{r}=1+|\frac{1}{p}-\frac{1}{2}|.$ An important tool in the formulation of our results will be the asymptotic behavior of $L^p$-norms for Hermite functions (see \cite{Thangavelu}).  There are complications to estimate norms of Hermite functions on different functions spaces (see e.g. \cite{Koch}), thus, our results can be not extended immediately to function spaces like Lebesgue spaces with exponent variable.  

Now, we present some references on the subject. Sufficient conditions for the  $r$-nuclearity of spectral multipliers associated to the harmonic oscillator, but, in modulation spaces and Wiener amalgam spaces have been considered by  J. Delgado, M. Ruzhansky and B. Wang in \cite{DRB,DRB2}. The Properties of these multipliers in $L^p$-spaces have been investigated in the references S. Bagchi, S. Thangavelu  \cite{BagchiThangavelu}, J. Epperson \cite{Epperson},  K. Stempak and J.L. Torrea \cite{stempak,stempak1,stempak2},  S. Thangavelu \cite{Thangavelu,Thangavelu2} and references therein. Hermite expansions for distributions can be found in B. Simon \cite{Simon}. The $r$-nuclearity and Grothendieck-Lidskii formulae for multipliers and other types of integral operators can be found in \cite{D3,DRB2}. Sufficient conditions for the nuclearity of  pseudo-differential operators on the torus can be found in \cite{DW,Ghaemi}. The references \cite{DR,DR1,DR3,DR5} and \cite{DRTk} include a complete study on the $r$-nuclearity of multipliers (and pseudo-differential operators) on compact Lie groups and more generally on compact manifolds. On Hilbert spaces the class of $r$-nuclear operators agrees with the Schatten-von Neumann class $S_{r}(H);$ in this context operators with integral kernel on Lebesgue spaces and, in particular, operators with  kernel acting of a special way with anharmonic oscillators of the form $E_a=-\Delta_x+|x|^a,$ $a>0,$ has been considered on  Schatten classes on $L^2(\mathbb{R}^n)$ in J. Delgado and M. Ruzhansky \cite{kernelcondition}.

Our main results will be presented in the next section. We end this paper by applying our trace formula to the the Hermite semigroup (or semigroup associated to the harmonic oscillator) $e^{-tH},$ $t>0$. We would like to thanks  C\'esar del Corral from Universidad de los Andes, Bogot\'a-Colombia, for discussions and his remarks about our work.

\section{$r$-Nuclear multipliers associated to the harmonic oscillator}\label{preliminaries}
\subsection{Preliminaries and notations: nuclear operators on Lebesgue spaces}
In this section we study  $r$-nuclear multipliers $T_m$  on Lebesgue spaces. Our criteria will be formulated in terms of the symbols $m.$ First, let us observe that every multiplier $T_m$ is an operator with kernel $K_m(x,y).$ In fact, straightforward computation show that 
\begin{align*}
T_mf(x) &=\int_{\mathbb{R}^n}K_m(x,y)f(y)dy,\,\,K_m(x,y):=\sum_{\nu\in\mathbb{N}^n_0}m(\nu)\phi_\nu(x)\phi_{\nu}(y)
\end{align*}
for every  $f\in\mathscr{D}(\mathbb{R}^n).$ In order to analyze the $r$-nuclearity of $T_m$ we study its kernel $K_m$ by using the following theorem (see J. Delgado \cite{Delgado,D2}).
\begin{theorem}\label{Theorem1} Let us consider $1\leq p_1,p_2<\infty,$ $0<r\leq 1$ and let $p_1'$ be such that $\frac{1}{p_1}+\frac{1}{p_1'}=1.$ An operator $T:L^p(\mu_1)\rightarrow L^p(\mu_2)$ is $r$-nuclear if and only if there exist sequences $(g_n)_n$ in $L^{p_2}(\mu_2),$ and $(h_n)$ in $L^{q_1}(\mu_1),$ such that
\begin{equation}
\sum_{n}\Vert g_n\Vert_{L^{p_2}}^r\Vert h_n\Vert^r_{L^{q_1} } <\infty,\textnormal{        and        }Tf(x)=\int(\sum_ng_{n}(x)h_n(y))f(y)d\mu_1(y),\textnormal{   a.e.w. }x,
\end{equation}
for every $f\in {L^{p_1}}(\mu_1).$ In this case, if $p_1=p_2$ $($\textnormal{see Section 3 of} \cite{Delgado}$)$ the nuclear trace of $T$ is given by
\begin{equation}\label{trace1}
\textnormal{Tr}(T):=\int\sum_{n}g_{n}(x)h_{n}(x)d\mu_1(x).
\end{equation}
\end{theorem}
If $1\leq p_1,p_2<\infty,$ by Theorem \ref{Theorem1}, we have that a multiplier $T_{m}:L^{p_1}(\mathbb{R}^n)\rightarrow L^{p_2}(\mathbb{R}^n)$ is $r$-nuclear if 
\begin{equation}\label{condition1}
s_{r}(m,p_1,p_2):=\sum_{\nu\in\mathbb{N}^n_0}|m(\nu)|^r\Vert \phi_\nu\Vert^r_{L^{p_2}(\mathbb{R}^n)}\Vert \phi_{\nu}\Vert^r_{L^{p_1'}(\mathbb{R}^n)}<\infty,
\end{equation}\label{trace11}
consequently, the nuclear trace of $T_{m}$ is given by
\begin{equation}\label{tr}
\textnormal{Tr}(T_m)=\int_{\mathbb{R}^n}\sum_{\nu\in\mathbb{N}^n}m(\nu)\phi_{\nu}(x)^2dx=\sum_{\nu\in\mathbb{N}^n}m(\nu)\int_{\mathbb{R}^n}\phi_{\nu}(x)^2dx=\sum_{\nu\in\mathbb{N}^n}m(\nu),
\end{equation}
where we have used that the $L^2$-norm of every $\phi_\nu$ is normalized.\\

Although the estimate \eqref{condition1} is a sufficient condition on $m$ in order to guarantee the $r$-nuclearity of $T_m,$  we want to study explicit conditions on $m,$ $p_{1},$ and $p_2$  in order that the series \eqref{condition1} converges. In order to find such conditions, we need to estimate the $L^p$-norms of Hermite functions which are eigenfunctions of the harmonic oscillator. This situation is analogue to one in the case of compact Lie groups where, in order to establish the $r$-nuclearity of multipliers,  it is important to estimate $L^p$-norms  of the eigenfunctions of the Laplacian (see D. Cardona \cite{Cardona} and J. Delgado and M. Ruzhansky \cite{DR}). \\

Our starting point is the following lemma on the asymptotic properties of $L^p$-norms for Hermite functions on $\mathbb{R}$ (see Lemma 4.5.2 of \cite{Thangavelu}).
\begin{lemma}\label{Lemma1}
Let us denote by $\phi_{\nu},$ $\nu\in\mathbb{N}_0,$ the one-dimensional Hermite functions. As $\nu\rightarrow \infty,$ these functions satisfy the estimates
\begin{itemize}
\item $\Vert \phi_\nu\Vert_{L^p(\mathbb{R})}\asymp \nu^{\frac{1}{2p}-\frac{1}{4}},$  $1\leq p\leq 2.$
\item $\Vert \phi_\nu\Vert_{L^p(\mathbb{R})}\asymp \nu^{-\frac{1}{8}}\ln(\nu),$ $p=4,$  
\item $\Vert \phi_\nu\Vert_{L^p(\mathbb{R})}\asymp \nu^{-\frac{1}{6p}-\frac{1}{12}},$ $4<p\leq \infty.$  
\end{itemize}
\end{lemma}
We need to extend the previous lemma to the case of $n$-dimensional Hermite functions, we explain the process of extension in the following remark. 
\begin{remark}
By the previous lemma, there exists $k\in\mathbb{N}$ large enough,  such that
\begin{itemize}
\item $\Vert \phi_\nu\Vert_{L^p(\mathbb{R})}\asymp \nu^{\frac{1}{2p}-\frac{1}{4}},$  $1\leq p<4,$  
\item $\Vert \phi_\nu\Vert_{L^p(\mathbb{R})}\asymp \nu^{-\frac{1}{8}}\ln(\nu),$ $p=4,$  
\item $\Vert \phi_\nu\Vert_{L^p(\mathbb{R})}\asymp \nu^{-\frac{1}{6p}-\frac{1}{12}},$ $4<p\leq \infty,$  
\end{itemize}
if $\nu\geq k,$ and for all $\nu\leq k,$ $\Vert \phi_\nu\Vert_{L^p(\mathbb{R})}\asymp \rho_k:=\Vert \phi_k\Vert_{L^p(\mathbb{R})} .$
Now, let us define the following $n+1$ subsets of $\mathbb{N}^n_0:$
\begin{itemize}
\item $I_0:=\{\nu\in\mathbb{N}^n_0:\nu_j> k\textnormal{  for all } 1\leq j\leq n\},$ 
\item  $I_s,$ for $s=1,2,\cdots, n,$  consists of all   $\nu\in\mathbb{N}^n_0$ such that  for some  $j_{1},\cdots,j_{s},$ $\nu_{j_1},\cdots \nu_{j_s}\leq k $  but, for $j\neq j_{1},j_2,\cdots j_s,$ $\nu_j>k.$ So, $I_s$ is the set of $n$-tuples of non negative integers with exactly $s$ entries less than or equal to $k.$
\end{itemize}
The sequence of subsets $I_0, I_1,\cdots I_{n},$ provides a partition of $\mathbb{N}^n_0$ which will be useful in order to estimate the series \eqref{condition1}. We end this remark by observing that for all $\nu\in I_s,$
\begin{align*}
\Vert \phi_\nu \Vert_{L^p(\mathbb{R}^n)}  =\prod_{v_j\leq k}\Vert \phi_{\nu_j}\Vert_{L^p(\mathbb{R})}\prod_{v_j> k}\Vert \phi_{\nu_j}\Vert_{L^p(\mathbb{R})} 
\asymp \Vert\phi_k \Vert_{L^p(\mathbb{R}^n)}^{s}\prod_{v_j> k}\Vert \phi_{\nu_j}\Vert_{L^p(\mathbb{R})}.
\end{align*} 
\end{remark}
With the notation of the preceding remark and by using Lemma \ref{Lemma1} we have immediately the following proposition.
\begin{proposition}
Let $s\in\mathbb{N}_0,$ and $\nu\in I_s\subset \mathbb{N}^n_0.$ The n-dimensional Hermite function $\phi_\nu$ satisfies the estimates
\begin{itemize}
\item $\Vert \phi_\nu\Vert_{L^p(\mathbb{R}^n)}\asymp {k}^{s(\frac{1}{2p}-\frac{1}{4})}\cdot ( \prod_{\nu_j>k}\nu_{j})^{\frac{1}{2p}-\frac{1}{4}},$  $1\leq p<4.$
\item $\Vert \phi_\nu\Vert_{L^p(\mathbb{R}^n)}\asymp k^{-\frac{s}{8}}(\ln(k))^s\cdot \prod_{\nu_j>k}[\nu_j^{-\frac{1}{8}}\ln(\nu_j)],$ $p=4,$  
\item $\Vert \phi_\nu\Vert_{L^p(\mathbb{R}^n)}\asymp k^{s(-\frac{1}{6p}-\frac{1}{12})}\cdot (\prod_{\nu_j>k} \nu_j)^{-\frac{1}{6p}-\frac{1}{12}},$ $4<p\leq \infty.$  
\end{itemize}
\end{proposition}
Extensions of Lemma \ref{Lemma1} have been obtained in \cite{Koch}, however, for our purposes will be sufficient such lemma and the preceding proposition.

\subsection{$r$-nuclearity of multipliers of the harmonic oscillator} By using the previous proposition, we study separately the $r$-nuclearity of multipliers by dividing our classification in the following three cases: $1\leq p_2<4$ and $1<p_1<\infty,$ $ p_2=4$ and $1<p_1<\infty,$ and $ 4<p_2\leq \infty,$ $1<p_1<\infty.$  We start with the following result.
\begin{theorem}
Let $1\leq p_2<4$ and $1<p_1<\infty.$ Let $T_m$ be a multiplier associated to the harmonic oscillator and $m$ its symbol. Then $T_{m}:L^{p_1}(\mathbb{R}^n)\rightarrow L^{p_2}(\mathbb{R}^n)$ extends to a $r$-nuclear operator if one of the following conditions holds:
\begin{itemize}
\item $\frac{4}{3}<p_1<\infty$ and
\begin{equation}
\varkappa(m,p_1,p_2):=\sum_{s=0}^n\sum_{\nu\in I_s}k^{\frac{sr}{2}(\frac{1}{p_2}-\frac{1}{p_1})} (\prod_{\nu_j>k}\nu_j )^{\frac{r}{2}(\frac{1}{p_2}-\frac{1}{p_1})}|m(\nu)|^r<\infty.
\end{equation}
\item  $p_1=\frac{4}{3}$ and
\begin{equation}
\varkappa(m,p_1,p_2):=\sum_{s=0}^n\sum_{\nu\in I_s}{k}^{\frac{sr}{2}(\frac{1}{p_2}-\frac{3}{4})}(\ln k)^{sr}\cdot  \prod_{\nu_j>k}[{\nu_j}^{\frac{r}{2}(\frac{1}{p_2}-\frac{3}{4})}(\ln(\nu_j))^r]|m(\nu)|^r<\infty.
\end{equation}
\item  $1<p_1<\frac{4}{3}$ and 
\begin{equation}
\varkappa(m,p_1,p_2):=\sum_{s=0}^n\sum_{\nu\in I_s} k^{\frac{sr}{2}(\frac{1}{p_2}+\frac{1}{3p_1}-1)}\cdot (\prod_{\nu_j>k} \nu_j)^{\frac{r}{2}(\frac{1}{p_2}+\frac{1}{3p_1}-1)}|m(\nu)|^r<\infty.
\end{equation} 
\end{itemize}In every case we have  $s_r(m,p_1,p_2) \asymp \varkappa(m,p_1,p_2).$
\end{theorem}
\begin{proof}
In view of Theorem \ref{Theorem1} we only  need to proof \eqref{condition1}. So, let us note that for $1\leq p_2<4$ and $\frac{4}{3}<p_1<\infty$ (which imply $1<p_1'<4$) we have the following estimate for  \eqref{condition1},
\begin{align*}
\sum_{\nu\in\mathbb{N}^n_0}|m(\nu)|^r & \Vert \phi_\nu\Vert^r_{L^{p_2}(\mathbb{R}^n)}  \Vert \phi_{\nu}\Vert^r_{L^{p_1'}(\mathbb{R}^n)} =\sum_{s=0}^{n}\sum_{\nu\in I_s}|m(\nu)|^r\Vert \phi_\nu\Vert^r_{L^{p_2}(\mathbb{R}^n)}\Vert \phi_{\nu}\Vert^r_{L^{p_1'}(\mathbb{R}^n)} \\
&\asymp \sum_{s=0}^{n}\sum_{\nu\in I_s}{k}^{sr(\frac{1}{2p_2}-\frac{1}{4})}\cdot ( \prod_{\nu_j>k}\nu_{j})^{r(\frac{1}{2p_2}-\frac{1}{4})}\cdot{k}^{sr(\frac{1}{2p_1'}-\frac{1}{4})}\cdot ( \prod_{\nu_j>k}\nu_{j})^{r(\frac{1}{2p_1'}-\frac{1}{4})}\\
&=\sum_{s=0}^{n}\sum_{\nu\in I_s}k^{\frac{sr}{2}(\frac{1}{p_2}-\frac{1}{p_1})} (\prod_{\nu_j>k}\nu_j) ^{\frac{r}{2}(\frac{1}{p_2}-\frac{1}{p_1})}|m(\nu)|^r<\infty.
\end{align*}
On the other hand, if we consider $1\leq p_2<4,$ $p_1=\frac{4}{3}$ (i.e $p_1'=4$) we obtain the following asymptotic expressions for \eqref{condition1},
\begin{align*}
\sum_{\nu\in\mathbb{N}^n_0}& |m(\nu)|^r  \Vert \phi_\nu\Vert^r_{L^{p_2}(\mathbb{R}^n)}  \Vert \phi_{\nu}\Vert^r_{L^{4}(\mathbb{R}^n)} =\sum_{s=0}^{n}\sum_{\nu\in I_s}|m(\nu)|^r\Vert \phi_\nu\Vert^r_{L^{p_2}(\mathbb{R}^n)}\Vert \phi_{\nu}\Vert^r_{L^{4}(\mathbb{R}^n)} \\
& \asymp \sum_{s=0}^{n}\sum_{\nu\in I_s}|m(\nu)|^r{k}^{sr(\frac{1}{2p_2}-\frac{1}{4})}\cdot ( \prod_{\nu_j>k}\nu_{j})^{r(\frac{1}{2p_2}-\frac{1}{4})}\cdot k^{-\frac{sr}{8}}(\ln(k))^{sr}\cdot \prod_{\nu_j>k}[\nu_j^{-\frac{r}{8}}(\ln(\nu_j))^r]\\
& =  \sum_{s=0}^{n} \sum_{\nu\in I_s}|m(\nu)|^r{k}^{\frac{sr}{2}(\frac{1}{p_2}-\frac{3}{4})}(\ln k)^{sr}\cdot  \prod_{\nu_j>k}[{\nu_j}^{\frac{r}{2}(\frac{1}{p_2}-\frac{3}{4})}(\ln(\nu_j))^r]<\infty.
\end{align*}
For the case where $1\leq p_2<4,$ $1<p_1<\frac{4}{3}$ (now, $4<p_1'<\infty$) we have
\begin{align*}
\sum_{\nu\in\mathbb{N}^n_0}& |m(\nu)|^r  \Vert \phi_\nu\Vert^r_{L^{p_2}(\mathbb{R}^n)}  \Vert \phi_{\nu}\Vert^r_{L^{p_1'}(\mathbb{R}^n)} =\sum_{s=0}^{n} \sum_{\nu\in I_s}|m(\nu)|^r\Vert \phi_\nu\Vert^r_{L^{p_2}(\mathbb{R}^n)}\Vert \phi_{\nu}\Vert^r_{L^{p_1'}(\mathbb{R}^n)} \\
& \asymp  \sum_{s=0}^{n} \sum_{\nu\in I_s}|m(\nu)|^r{k}^{sr(\frac{1}{2p_2}-\frac{1}{4})}\cdot ( \prod_{\nu_j>k}\nu_{j})^{r(\frac{1}{2p_2}-\frac{1}{4})}\cdot k^{sr(-\frac{1}{6p_1'}-\frac{1}{12})}\cdot (\prod_{\nu_j>k} \nu_j)^{-\frac{r}{6p_1'}-\frac{1}{12}}\\
& = \sum_{s=0}^{n} \sum_{\nu\in I_s} k^{\frac{sr}{2}(\frac{1}{p_2}+\frac{1}{3p_1}-1)}\cdot (\prod_{\nu_j>k} \nu_j)^{\frac{r}{2}(\frac{1}{p_2}+\frac{1}{3p_1}-1)}|m(\nu)|^r<\infty.
\end{align*}
thus, we conclude the proof.
\end{proof}

Now, we classify the $r$-nuclearity of multipliers when $p_2=4.$
\begin{theorem}
Let  $1<p_1<\infty.$ Let $T_m$ be a multiplier associated to the harmonic oscillator and $m$ its symbol. Then $T_{m}:L^{p_1}(\mathbb{R}^n)\rightarrow L^{4}(\mathbb{R}^n)$ extends to a $r$-nuclear operator if  one of the following conditions holds:
\begin{itemize}
\item $\frac{4}{3}<p_1<\infty$ and 
\begin{equation}
\varkappa(m,p_1,p_2):=\sum_{s=0}^{n} \sum_{\nu\in {I}_s}k^{\frac{sr}{2}(\frac{1}{4}-\frac{1}{p_1})}(\ln (k))^{sr}\prod_{\nu_j>k}[(\ln(\nu_j))^r\nu_j^{\frac{r}{2}(\frac{1}{4}-\frac{1}{p_1})}]|m(\nu)|^r<\infty.
\end{equation}
\item $p_1=\frac{4}{3}$ and 
\begin{equation} \varkappa(m,p_1,p_2):=\sum_{s=0}^{n}
\sum_{\nu\in {I}_s}k^{-\frac{sr}{4}}(\ln k)^{2sr}\prod_{\nu_j>k}[\nu_j^{-\frac{r}{4}}(\ln \nu_j)^{2r}]\cdot |m(\nu)|^r<\infty.
\end{equation}
\item $1<p_1<\frac{4}{3}$ and 
\begin{equation}
\varkappa(m,p_1,p_2):=\sum_{s=0}^{n} \sum_{\nu\in I_s}k^{\frac{sr}{6}(\frac{1}{p_1}-\frac{9}{4})}(\ln(k))^{sr}\prod_{\nu_j>k}[\nu_j^{\frac{r}{6}(\frac{1}{p_1}-\frac{9}{4})}\ln(\nu_j)^r  ]\cdot |m(\nu)|^r<\infty.
\end{equation}
\end{itemize} In every case we have  $s_r(m,p_1,p_2) \asymp \varkappa(m,p_1,p_2).$
\end{theorem}
\begin{proof}
First, we consider the case $\frac{4}{3}<p_1<\infty,$ ($1<p_1'<4.$) In fact, we have
\begin{align*}
\sum_{\nu\in\mathbb{N}^n_0}& |m(\nu)|^r  \Vert \phi_\nu\Vert^r_{L^{4}(\mathbb{R}^n)}  \Vert \phi_{\nu}\Vert^r_{L^{p_1'}(\mathbb{R}^n)} =\sum_{s=0}^{n} \sum_{\nu\in {I}_s} |m(\nu)|^r  \Vert \phi_\nu\Vert^r_{L^{4}(\mathbb{R}^n)}  \Vert \phi_{\nu}\Vert^r_{L^{p_1'}(\mathbb{R}^n)}\\
&\asymp \sum_{s=0}^{n} \sum_{\nu\in {I}_s} |m(\nu)|^r k^{-\frac{sr}{8}}(\ln(k))^{sr}\cdot \prod_{\nu_j>k}[\nu_j^{-\frac{r}{8}}(\ln(\nu_j))^r]\cdot {k}^{sr(\frac{1}{2p_1'}-\frac{1}{4})}\cdot ( \prod_{\nu_j>k}\nu_{j})^{r(\frac{1}{2p_1'}-\frac{1}{4})}\\
&= \sum_{s=0}^{n} \sum_{\nu\in {I}_s} k^{\frac{sr}{2}(\frac{1}{4}-\frac{1}{p_1})}(\ln (k))^{sr}\prod_{\nu_j>k}[(\ln(\nu_j))^r\nu_j^{\frac{r}{2}(\frac{1}{4}-\frac{1}{p_1})}]|m(\nu)|^r<\infty.
\end{align*}
For the case where $p_1=\frac{4}{3},$ (i.e. $p_1'=4$), we have
\begin{align*}
\sum_{\nu\in\mathbb{N}^n_0}& |m(\nu)|^r  \Vert \phi_\nu\Vert^r_{L^{4}(\mathbb{R}^n)}  \Vert \phi_{\nu}\Vert^r_{L^{p_1'}(\mathbb{R}^n)} =\sum_{s=0}^{n} \sum_{\nu\in {I}_s} |m(\nu)|^r  \Vert \phi_\nu\Vert^{2r}_{L^{4}(\mathbb{R}^n)}  \\
&\asymp \sum_{s=0}^{n} \sum_{\nu\in {I}_s}  k^{-\frac{sr}{4}}(\ln(k))^{2sr}\cdot \prod_{\nu_j>k}[\nu_j^{-\frac{r}{4}}(\ln(\nu_j))^{2r}]\cdot|m(\nu)|^r <\infty.
\end{align*}
Finally, if we consider $1<p_1<\frac{4}{3},$ then $4<p_1'<\infty,$ and we estimate \eqref{condition1} as follows
\begin{align*}
\sum_{\nu\in\mathbb{N}^n_0}& |m(\nu)|^r  \Vert \phi_\nu\Vert^r_{L^{4}(\mathbb{R}^n)}  \Vert \phi_{\nu}\Vert^r_{L^{p_1'}(\mathbb{R}^n)} = \sum_{s=0}^{n} \sum_{\nu\in {I}_s} |m(\nu)|^r\Vert \phi_\nu\Vert^r_{L^{4}(\mathbb{R}^n)}  \Vert \phi_{\nu}\Vert^r_{L^{p_1'}(\mathbb{R}^n)}\\
&\asymp \sum_{s=0}^{n} \sum_{\nu\in {I}_s} |m(\nu)|^r k^{-\frac{sr}{8}}(\ln(k))^{sr}\cdot \prod_{\nu_j>k}[\nu_j^{-\frac{r}{8}}(\ln(\nu_j))^r]\cdot k^{sr(-\frac{1}{6p_1'}-\frac{1}{12})}\cdot (\prod_{\nu_j>k} \nu_j)^{-\frac{r}{6p_1'}-\frac{1}{12}}\\
&\asymp \sum_{s=0}^{n} \sum_{\nu\in I_s}k^{\frac{sr}{6}(\frac{1}{p_1}-\frac{9}{4})}(\ln(k))^{sr}\prod_{\nu_j>k}[\nu_j^{\frac{r}{6}(\frac{1}{p_1}-\frac{9}{4})}\ln(\nu_j)^r  ]\cdot |m(\nu)|^r<\infty.
\end{align*}
thus, we end the proof.
\end{proof}

We end this chapter with some conditions for the $r$-nuclearity of multipliers when $4<p_2\leq \infty.$ We omit the proof because we only need to repeat the arguments above on asymptotic properties of Hermite functions.
\begin{theorem}
Let  $4<p_2\leq \infty$ and  $1<p_1<\infty.$ Let $T_m$ be a multiplier associated to the harmonic oscillator and $m$ its symbol. Then $T_{m}:L^{p_1}(\mathbb{R}^n)\rightarrow L^{p_2}(\mathbb{R}^n)$ extends to a $r$-nuclear operator if one of the following conditions holds:
\begin{itemize}
\item $\frac{4}{3}<p_1<\infty $ and 
\begin{equation}
\varkappa(m,p_1,p_2):=\sum_{s=0}^{n} \sum_{\nu\in I_s}k^{\frac{sr}{2}(\frac{1}{3p_2'}-\frac{1}{p_1})}(\prod_{\nu_j>k}\nu_j)^{\frac{r}{2}(\frac{1}{3p_2'}-\frac{1}{p_1})}|m(\nu)|^r<\infty.
\end{equation}
\item $p_1=\frac{4}{3}$ and 
\begin{equation}
\varkappa(m,p_1,p_2):=\sum_{s=0}^{n} \sum_{\nu\in I_s}k^{-\frac{sr}{6}(\frac{1}{p_2}+\frac{5}{4})}(\ln(k))^{sr}\prod_{\nu_j>k}[  \nu_j^{-\frac{r}{6}(\frac{1}{p_2}+\frac{5}{4})}(\ln(\nu_j))^{r}  ]|m(\nu)|^r<\infty.
\end{equation}
\item $1<p_1<\frac{4}{3}$ and 
\begin{equation}
\varkappa(m,p_1,p_2):=\sum_{s=0}^{n} \sum_{\nu\in I_s}k^{\frac{sr}{6}(\frac{1}{p_1}-\frac{1}{p_2}-2)}\cdot (\prod_{\nu_j>k}\nu_j)^{\frac{r}{6}(\frac{1}{p_1}-\frac{1}{p_2}-2)}|m(\nu)|^r<\infty.
\end{equation} 
\end{itemize} In every case we have $s_r(m,p_1,p_2) \asymp \varkappa(m,p_1,p_2).$
\end{theorem}

\section{ Spectral and nuclear traces of multipliers of the harmonic oscillator}\label{mho}
\subsection{Trace formulae for multipliers} If  $T:E\rightarrow E$ is  $r$-nuclear, with the Banach space $E$ satisfying the approximation property (see \cite{GRO}),  then 
there exist  sequences $(e_n ')_{n\in\mathbb{N}_0}$ in $ E'$ (the dual space of $E$) and $(y_n)_{n\in\mathbb{N}_0}$ in $E$ such that
\begin{equation}\label{nuc2}
Tf=\sum_{n\in\mathbb{N}_0} e_n'(f)y_n,\,\,\,\,\,\textnormal{and}\,\,\,\,\,
\sum_{n\in\mathbb{N}_0} \Vert e_n' \Vert^r_{E'}\Vert y_n \Vert^r_{F}<\infty.
\end{equation}
In this case the nuclear trace of $T$ is given by
$
\textnormal{Tr}(T)=\sum_{n\in\mathbb{N}^n_0}e_n'(f_n).
$
$L^p$-spaces have the approximation property and as consequence we can compute the nuclear trace of every $r$-nuclear multiplier $T_m:L^{p}(\mathbb{R}^n)\rightarrow L^{p}(\mathbb{R}^n).$ In fact, by \eqref{tr} we have
\begin{equation}
\textnormal{Tr}(T_m)=\sum_{\nu\in\mathbb{N}^n}m(\nu).
\end{equation}

It was proved by Grothendieck that the nuclear trace of an $r$-nuclear operator  on a Banach space coincides with the spectral trace, provided that $0<r\leq \frac{2}{3}.$ For $\frac{2}{3}\leq r\leq 1$ we recall the following result (see \cite{O}).
\begin{theorem} Let $T:L^p(\mu)\rightarrow L^p(\mu)$ be a $r$-nuclear operator as in \eqref{nuc2}. If $\frac{1}{r}=1+|\frac{1}{p}-\frac{1}{2}|,$ then, 
\begin{equation}
\textnormal{Tr}(T):=\sum_{n\in\mathbb{N}^n_0}e_n'(f_n)=\sum_{n}\lambda_n(T)
\end{equation}
where $\lambda_n(T),$ $n\in\mathbb{N}$ is the sequence of eigenvalues of $T$ with multiplicities taken into account. 
\end{theorem}
As an immediate consequence of the preceding theorem, if $T_m:L^p(\mathbb{R}^n)\rightarrow L^p(\mathbb{R}^n)$ is a $r$-nuclear  multiplier and $\frac{1}{r}=1+|\frac{1}{p}-\frac{1}{2}|$ then, 
\begin{equation}
\textnormal{Tr}(T_m)=\sum_{\nu\in\mathbb{N}^n}m(\nu)
=\sum_{n}\lambda_n(T),
\end{equation}
where $\lambda_n(T),$ $n\in\mathbb{N}$ is the sequence of eigenvalues of $T_m$ with multiplicities taken into account. 

\subsection{The Hermite semigroup} Let us consider the Hermite functions $\phi_\nu,$  $\nu\in\mathbb{N}^n_0.$ If $P_{k},$ $k\in\mathbb{N}_0,$ is the projection on $L^{2}(\mathbb{R}^n)$ given by
\begin{equation}
P_{k}f(x):=\sum_{|\nu|=k}\widehat{f}(\phi_\nu)\phi_\nu(x),
\end{equation}
 then,  the Hermite semigroup (semigroup associated to the harmonic oscillator) $e^{-tH},$ $t>0$ is given by
 \begin{equation}
 e^{-tH}f(x)=\sum_{k=0}^{\infty}e^{-t(2k+n)}P_{k}f(x).
 \end{equation}
 For every $t>0,$ the operator $e^{-tH}$ has Schwartz kernel given by
 \begin{equation}\label{kernelt}
 K_{t}(x,y)=\sum_{\nu\in\mathbb{N}^n_0}e^{-t(2|\nu|+n)}\phi_{\nu}(x)\phi_\nu(y).
 \end{equation}
 In view of Mehler's formula (see Thangavelu \cite{Thangavelu}) the above series can be summed up and we obtain
$
 K_{t}(x,y)=(2\pi)^{-\frac{n}{2}}\sinh(2t)^{-\frac{n}{2}}e^{-(\frac{1}{2}|x|^2+|y|^2)\coth(2t)+x\cdot y\cdot \textnormal{csch}(2t))}.
 $
 The symbol of the operator $e^{-tH},$  $m_{t}(\nu)=e^{-t(2|\nu|+n)},$ has exponential decaying and  we can deduce that  $e^{-tH}$ extends to a $r$-nuclear on every $L^p(\mathbb{R}^n),$ $1\leq p<\infty,$ (this can be verified if we insert the symbol $m_t$ on every condition of the previous subsection) with trace (see \eqref{trace1}) given by
 \begin{equation}
 \textnormal{Tr}(e^{-tH})=\sum_{\nu\in\mathbb{N}^n_0}e^{-t(2|\nu|+n)}=(e^{t}-e^{-t})^{-n}.
 \end{equation}
 The trace of $e^{-tH}$  can also be computed by applying \eqref{trace11}, in fact 
 \begin{align*}
 \textnormal{Tr}(e^{-tH}) &=\int_{\mathbb{R}^n}K_t(x,x)dx =(2\pi)^{-\frac{n}{2}}\int_{\mathbb{R}^n} \sinh(2t)^{-\frac{n}{2}}e^{-|x|^2(\coth(2t)-\textnormal{csch}(2t))}dx\\
 &= (2\pi)^{-\frac{n}{2}}\sinh(2t)^{-\frac{n}{2}}(\pi/(\coth(2t)-\textnormal{csch}(2t)))^{\frac{n}{2}}\\
 &=(e^{t}-e^{-t})^{-n},
 \end{align*}
 where we have use the known formula: $\int_{\mathbb{R}}e^{-ax^2}dx=(\pi/a)^{\frac{1}{2}}.$

\bibliographystyle{amsplain}

\end{document}